\documentclass[reqno, 12pt]{amsart}
\pdfoutput=1
\makeatletter
\let\origsection=\section 
\def\section{\@ifstar{\origsection*}{\mysection}} 
\def\mysection{\@startsection{section}{1}\z@{.7\linespacing\@plus\linespacing}{.5\linespacing}{\normalfont\scshape\centering\S}}
\makeatother        
\usepackage{amsmath,amssymb,amsthm}  
\usepackage{mathabx}\changenotsign  
\usepackage{mathrsfs} 
\usepackage{dsfont}

\usepackage{xcolor}  	
\usepackage[backref]{hyperref}
\hypersetup{
	colorlinks,
    linkcolor={red!60!black},
    citecolor={green!60!black},
    urlcolor={blue!60!black},
}

\usepackage[open,openlevel=2,atend]{bookmark}

\usepackage[abbrev,msc-links,backrefs]{amsrefs} 
\usepackage{doi}

\renewcommand{\PrintDOI}[1]{\doi{#1}}

\usepackage[T1]{fontenc}
\usepackage{lmodern}
\usepackage[babel]{microtype}
\usepackage[english]{babel}

\numberwithin{equation}{section}

\usepackage{geometry}
\usepackage{picins}
\usepackage{graphicx}              
\usepackage{caption}
\usepackage{subcaption}
\usepackage{graphics}

\usepackage{enumitem}

\def\alabel{\upshape({\itshape \alph*\,})}

\let\polishlcross=\l
\def\l{\ifmmode\ell\else\polishlcross\fi}

\def\qand{\quad\text{and}\quad}
\def\qqand{\qquad\text{and}\qquad}

\let\setminus=\smallsetminus

\usepackage{accents}
\newcommand{\seq}[1]{\accentset{\rightharpoonup}{#1}}

\makeatletter
\def\moverlay{\mathpalette\mov@rlay}
\def\mov@rlay#1#2{\leavevmode\vtop{
   \baselineskip\z@skip \lineskiplimit-\maxdimen
   \ialign{\hfil$\m@th#1##$\hfil\cr#2\crcr}}}
\newcommand{\charfusion}[3][\mathord]{
    #1{\ifx#1\mathop\vphantom{#2}\fi
        \mathpalette\mov@rlay{#2\cr#3}
      }
    \ifx#1\mathop\expandafter\displaylimits\fi}
\makeatother

\newcommand{\dcup}{\charfusion[\mathbin]{\cup}{\cdot}}
\newcommand{\bigdcup}{\charfusion[\mathop]{\bigcup}{\cdot}}

\let\eps=\varepsilon
\let\theta=\vartheta
\let\rho=\varrho
\let\phi=\varphi

\def\NN{\mathds N}

\def\cF{{\mathcal F}}

\def\cS{{\mathcal S}}

\def\cK{{\mathcal K}}

\def\cW{{\mathcal W}}

\theoremstyle{plain}
\newtheorem{thm}{Theorem}[section]
\newtheorem{fact}[thm]{Fact}
\newtheorem{prop}[thm]{Proposition}

\newtheorem{cor}[thm]{Corollary}
\newtheorem{lemma}[thm]{Lemma}

\theoremstyle{definition}

\newtheorem{dfn}[thm]{Definition}

\DeclareFontFamily{U}  {MnSymbolC}{}
\DeclareSymbolFont{MnSyC}         {U}  {MnSymbolC}{m}{n}
\DeclareFontShape{U}{MnSymbolC}{m}{n}{
    <-6>  MnSymbolC5
   <6-7>  MnSymbolC6
   <7-8>  MnSymbolC7
   <8-9>  MnSymbolC8
   <9-10> MnSymbolC9
  <10-12> MnSymbolC10
  <12->   MnSymbolC12}{}
\DeclareMathSymbol{\powerset}{\mathord}{MnSyC}{180}
\DeclareMathSymbol{\suptriangle}{\mathord}{MnSyC}{72}
\DeclareMathSymbol{\muptriangle}{\mathord}{MnSyC}{80}
\DeclareMathSymbol{\luptriangle}{\mathord}{MnSyC}{84}

\def\triag{{\muptriangle}}

\def\tri{\triangle}

\def\t{t}

\begin{document}
\title[Forcing quasirandomness with triangles]{Forcing quasirandomness with triangles}
\thanks{Research of both author was supported by ERC Consolidator Grant 724903.}

\author[Christian Reiher]{Christian Reiher}
\author[Mathias Schacht]{Mathias Schacht}
\address{Fachbereich Mathematik, Universit\"at Hamburg, Hamburg, Germany}
\email{Christian.Reiher@uni-hamburg.de}
\email{schacht@math.uni-hamburg.de}

\begin{abstract}
	We study \emph{forcing pairs} for \emph{quasirandom graphs}.
	Chung, Graham, and Wilson initiated the study of families~$\cF$
	of graphs with the property that if a large graph~$G$ has 
	approximately homomorphism density $p^{e(F)}$ for some fixed $p\in(0,1]$ for every $F\in \cF$,
	then~$G$ is quasirandom with density~$p$. Such families $\cF$ are said to be \emph{forcing}.
	Several forcing families were found over the last three decades and characterising all 
	bipartite graphs $F$ such that~$(K_2,F)$ is a forcing pair is 
	a well-known open problem in the area of quasirandom graphs, which is closely 
	related to Sidorenko's conjecture. In fact, most of 
	the known forcing families involve bipartite graphs only.
	
	We consider forcing pairs containing the triangle~$K_3$. In particular, we show that  
	if~$(K_2,F)$ is a forcing pair, then 
	so is $(K_3,F^\triag)$, where $F^\triag$ is obtained from~$F$
	by replacing every edge of $F$ by a triangle (each of which introduces a new vertex).
	For the proof we first show that~$(K_3,C_4^\triag)$ is a forcing pair, which 
	strengthens related results of Simonovits and S\'os
	and of Conlon et al.
\end{abstract} 

\maketitle

\section{Introduction}\label{sec:intro}
The systematic study of quasirandom graphs was initiated by Thomason~\cites{Th87a,Th87b} 
and Chung, Graham, and Wilson~\cite{CGW89} and over the last 30 years many generalisations 
and extensions to directed graphs~\cite{Gr13}, tournaments~\cite{CG91}, 
hypergraphs~\cites{ACHPS,Ch90,Ch12,CG90,CHPS,HaTh89,KRS02,LM15,To17}, set systems~\cite{CG91b}, 
permutations~\cite{Co04}, groups~\cite{Gow08}, subsets of cyclic groups and finite fields~\cites{Wi72,CG92}, and 
sparse graphs~\cites{CG02,CFZ14,KR03a} were established by several researchers (see, e.g., the surveys~\cites{KR03,KS06} for a more detailed discussion). Roughly speaking, 
a given discrete structure is quasirandom if it shares important properties with a 
``truly random'' structure of the same size. In the context of graphs this is made precise by 
mimicking the uniform edge distribution of the random graph~$G(n,p)$.
\begin{dfn} For $\eps>0$ and $p\in(0,1]$
	we say a graph $G=(V,E)$ is \emph{$(\eps,p)$-quasirandom}, if for all subsets $X$, $Y\subseteq V$
	we have 
	\[
		\big|e_G(X,Y)-p\,|X|\,|Y|\big|\leq \eps |V|^2\,,
	\]
	where edges contained in the intersection $X\cap Y$ are counted twice in $e_G(X,Y)$.
\end{dfn}
In the context of quasirandom graphs we often consider sequences of graphs $\seq{G}=(G_n)_{n\in\NN}$
with $|V(G_n)|\to\infty$. Then we may say that the sequence \emph{$\seq{G}$ is 
$p$-quasirandom} if for every~$\eps>0$ all but finitely many members of the sequence 
are $(\eps,p)$-quasirandom. For a simpler discussion, we sometimes 
say that a graph~$G$ is $p$-quasirandom without any reference to the sequence, by which 
we mean that some sufficiently large graph~$G$ is $(\eps,p)$-quasirandom for some small 
unspecified~$\eps>0$. 
If the density~$p$ is of no particular importance, then we may just say $G$ is quasirandom.

A large part of the theory of quasirandom graphs concerns equivalent 
characterisations of~\mbox{$p$-quasirandom} graph sequences. Early results 
in that direction implicitly appeared in~\cites{Al86,AlCh88,FRW88,Ro86} and 
Chung, Graham, and Wilson~\cite{CGW89} gave six alternative characterisation. 
Since then many more such characterisations were found (see,
e.g.,~\cites{CG92a,CHPS,HuLe12,JaSo15,Sh08,SY10,SY12,SiSo91,SiSo97,SiSo03,SkTh04,Y10}).

Here we focus on characterisations that rely on the densities of graph homomorphisms of 
given graphs~$F$ into large graphs~$G$. We denote by $\hom(F,G)$ the number of graph homomorphisms 
from $F$ into~$G$ and the \emph{homomorphism density  $\t(F,G)$} is defined by 
\[
	\t(F,G)=\frac{\hom(F,G)}{|V(G)|^{|V(F)|}}\,.
\]
Let us recall that a pair of graphs $(F_1,F_2)$ is said to be {\it forcing} 
if for every $p\in(0,1]$ and $\eps>0$ there is some $\delta>0$ such that the following holds: 
if a graph $G$ satisfies 
\begin{equation}\label{eq:def:forcing}
	\t(F_1, G)\geq (1-\delta)p^{e(F_1)}
	\qqand
	\t(F_2, G)\leq (1+\delta)p^{e(F_2)}
\end{equation}
then it is $(\eps,p)$-quasirandom.
This notion goes back to~\cite{CGW89} and has frequently been discussed in the literature. 
The most classical example of such a family is the pair $(K_2, C_4)$. 
The statement that the pair $(K_2, F)$ is forcing for every bipartite graph $F$ that 
is not a forest, called the {\it forcing conjecture}, can be traced back to 
Skokan and Thoma~\cite{SkTh04} (see also~\cite{CFS10}). 
It has been the subject of intensive study that led to its verification in various cases. 
For the most recent contributions to the forcing conjecture and the closely related conjectures of 
Erd\H os and Simonovits~\cite{Si84} and of Sidorenko~\cite{Si92} we refer to~\cites{CFS10,CKLL,CoLe17,Ha10,Korea,LSz,Szegedy}.

Until  recently all known forcing pairs contained bipartite graphs only.
In fact, already Chung, Graham, and Wilson~\cite{CGW89} noted that $(K_2,K_3)$ (and also $(K_{1,2},K_3)$)
is not a forcing pair, by giving an example of $n$-vertex graphs $G$ with all vertices having degree close 
to~$n/2$ and with $t(K_3,G)\approx 1/8$, but containing independent sets and cliques of size~$\lfloor n/4\rfloor$.
However, it was shown by Simonovits and S\'os~\cite{SiSo97} that such a situation can be avoided by appealing 
to the \emph{hereditary} nature of quasirandom graphs, i.e., if $G=(V,E)$ is $p$-quasirandom, then the induced 
subgraphs~$G[U]$ are $p$-quasirandom for linear sized subsets~$U\subseteq V$. Simonovits and S\'os 
then showed that requiring 
\begin{equation}\label{eq:lforcing}
	\t(F, G[U])=\Big(1\pm \delta \tfrac{|V|}{|U|}\Big)p^{e(F)}
\end{equation}
for a given graph $F$ with at least one edge and for all $U\subseteq V$
forces $G$ to be $p$-quasirandom. Recently, Conlon, H\`an, Person, and Schacht~\cite{CHPS} 
(see also~\cite{HPS11}) observed that condition~\eqref{eq:lforcing} 
gives rise to a forcing pair $(F,M_F)$
for an appropriate graph $M_F$ depending on~$F$.

\parpic[r]{\includegraphics[width=4cm]{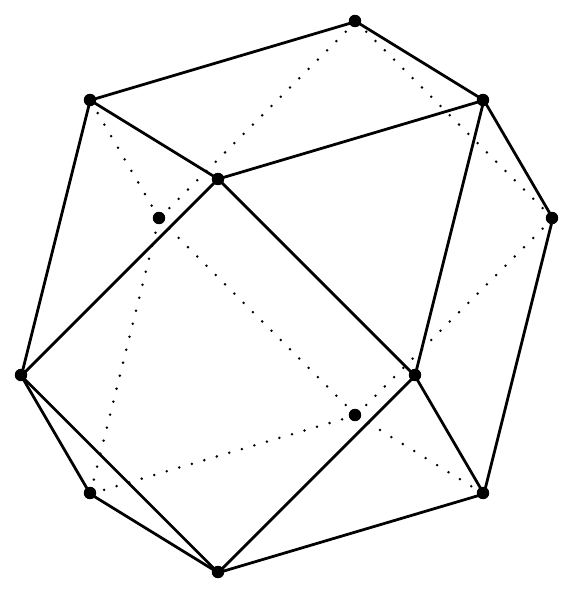}}
We study forcing pairs involving triangles. For that case it was shown in~\cite{CHPS}
that the pair $(K_3, M)$ is forcing, where $M$ denotes the line graph of the 
$3$-dimensional Boolean cube, depicted on the right. 
The idea behind the proof is roughly as follows: Three successive applications of the 
Cauchy--Schwarz inequality yield $\t(M, G)\ge \t(K_3, G)^8$ for any graph~$G=(V,E)$. 
On the other hand, the assumption~\eqref{eq:def:forcing} for $(K_3,M)$ tells us $t(K_3, G)\gtrsim p^3$
and $\t(M, G)\lesssim p^{24}$ for some real $p\in(0,1]$. 
Consequently, an approximate equality must hold in each of these three steps, and it may be argued 
that this is in turn only possible if for all subsets $A$, $B$, $C\subseteq V$ we have 
\[
	\tri(A, B, C)\approx p^3\,|A|\,|B|\,|C|\,,
\] 
where $\tri(A, B, C)$ denotes the number of triangles with a vertex in $A$, 
a vertex in~$B$, and a vertex in~$C$. This yields~\eqref{eq:lforcing}
for $F=K_3$ and the Simonovits--S\'os theorem implies that $G$ is $p$-quasirandom.

\parpic[l]{\includegraphics[width=4cm]{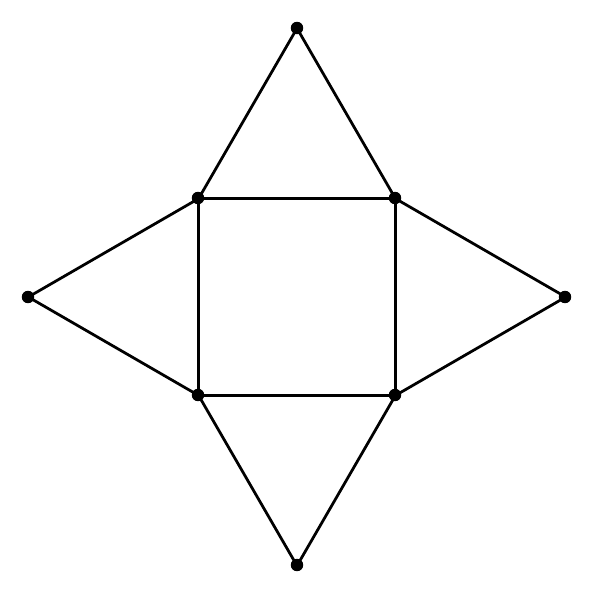}} Our main result shows that 
the same effect as above can be achieved with two applications of the 
Cauchy--Schwarz inequality only. This implies that the pair $(K_3, C_4^\triag)$ is forcing, 
where $C_4^\triag$ is obtained from the $4$-cycle $C_4$ where every edge is replaced by a triangle 
(each of which introduces a new vertex), 
i.e., the graph shown on the left. 
As we shall explain in more detail below, we have $\t(C_4^\triag, G)\ge \t(K_3, G)^4$ for all graphs $G$, 
and if approximate equality holds for some graph~$G$, then it satisfies the assumption of 
the following theorem, which weakens the assumption of the Simonovits--S\'os theorem in the triangle case. 
\begin{thm}\label{thm:2sets}
For every $p\in(0,1]$ and $\eps>0$ there is an $\eta>0$ such that any graph ${G=(V,E)}$ 
satisfying 
\begin{equation}\label{eq:2sets}
	\tri(A, B, V)=p^3\,|A|\,|B|\,|V|\pm \eta\,|V|^3
\end{equation}
for all $A, B\subseteq V$
is $(\eps,p)$-quasirandom.
\end{thm}

The following corollary renders the aforementioned connection between condition~\eqref{eq:2sets} 
and~$\t(C_4^\triag, G)\approx \t(K_3, G)^4$ and strengthens the result of Conlon et al. 
that $(K_3,M)$ is a forcing pair.
\begin{cor}\label{thm:K3Nforcing}
The pair $(K_3, C_4^{\triag})$ is forcing.
\end{cor}

As it turns out Corollary~\ref{thm:K3Nforcing} applied in the context of weighted graphs 
allows the following general result, which is our main result on forcing pairs involving triangles.

\begin{cor}\label{thm:Hauptsatz}
	If~$(K_2, F)$ is a forcing pair, then so is~$(K_3,F^\triag)$.
\end{cor}

\subsection*{Organisation} 
We prove Theorem~\ref{thm:2sets} in Section~\ref{sec:2sets}
and Corollary~\ref{thm:K3Nforcing} in Section~\ref{sec:forcing}. 
In Section~\ref{sec:ana} we switch to the analytical language of graphons
and prove Corollary~\ref{thm:Hauptsatz}.
We conclude by recording some further observations and problems for future research
in Section~\ref{sec:conc}.

\section{The two sets condition} \label{sec:2sets}
The proof of Theorem~\ref{thm:2sets} is based on the regularity method for graphs. 
This means that we use a regularity lemma and a counting lemma in order 
to reduce the problem at hand to a somewhat different one that speaks about a certain 
``reduced graph.'' 
In the present situation we need to conceive this reduced graph 
as a {\it weighted graph}. Such objects may also be regarded as symmetric matrices 
with entries from the unit interval. The precise statement we shall require is stated as 
Lemma~\ref{lem:2sets-reduced} below. The interested reader may check that this lemma could
conversely also be deduced from Theorem~\ref{thm:2sets}. 

\begin{lemma}\label{lem:2sets-reduced}
Given any real numbers $p\in(0,1]$ and $\eps>0$ there is a real $\delta>0$ such that the 
following holds: Let $(d_{ij})_{i,j\in[t]}\in[0,1]^{t\times t}$ be a symmetric matrix 
such that for all distinct indices $i, j\in [t]$ we have 
\begin{equation}\label{eq:dijk}
	d_{ij}\sum_{k\in[t]}d_{ik}d_{jk}=(p^3\pm \delta)t\,.
\end{equation}
Then $d_{ij}=p\pm\eps$ holds for all $i,j\in[t]$.
\end{lemma}  

\begin{proof} Throughout the proof we work with the hierarchy $\delta\ll\rho\ll p,\eps$ for some 
auxiliary chosen constant $\rho$, where we write $\alpha\ll \beta$ to signify that $\alpha$ will be chosen 
sufficiently small depending on $\beta$.
\parpic[l]{\includegraphics[width=3.6cm]{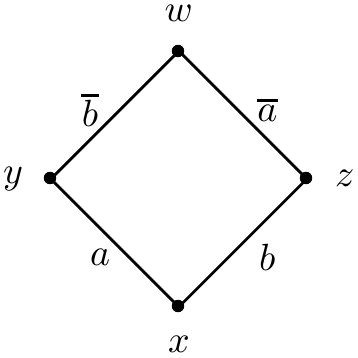}\smallskip} 
Since the sum on the left-hand side 
of~\eqref{eq:dijk} is at most $t$, we have 
${d_{ij}\ge p^3-\delta\ge p^3/2}$ for all $i, j\in[t]$. 
The main idea is to choose three indices $x, y, z\in [t]$ such that the difference $d_{xy}-d_{xz}$ 
is as large as possible. In the picture to the left, these indices $x$, $y$, and $z$ are
represented as ``vertices'' and the labels $a$ and $b$ attached to the ``edges''~$xy$ and 
$xz$ indicate that we set $a=d_{xy}$ and $b=d_{xz}$, respectively. The vertex $w$, which is 
shown there as well, will be chosen later in the argument.  

The maximality of $a-b$ entails
\begin{equation}\label{eq:diyz}
	d_{yi}\ge b \,\,\, \text{and} \,\,\, d_{zi}\le a \,\,\, \text{for all } i\in [t]\,.
\end{equation}
By~\eqref{eq:dijk} we have
\[
	a\cdot\sum_{i\in[t]}d_{xi}d_{yi}\le (p^3+\delta)t
	\quad \text{ and } \quad
	b\cdot\sum_{i\in[t]}d_{xi}d_{zi}\ge (p^3-\delta)t\,,
\]
whence
\[
	\sum_{i\in[t]}d_{xi}(ad_{yi}-bd_{zi})\le 2\delta t\,.
\]
Now~\eqref{eq:diyz} yields 
\[
	ad_{yi}-bd_{zi}=a(d_{yi}-b)+b(a-d_{zi})\ge \big((d_{yi}-b)+(a-d_{zi})\big)p^3/2\ge 0
\]
for all $i\in [t]$, so it follows that
\begin{equation}\label{eq:dsums}
	\sum_{i\in[t]}(d_{yi}-b)+\sum_{i\in[t]}(a-d_{zi})\le 8p^{-6}\delta t\,,
\end{equation}
where neither sum on the left-hand side is negative. Thus there is an index $w\in[t]$ with 
$(d_{wy}-b)+(a-d_{wz})\le 8p^{-6}\delta$. In particular, the numbers $\overline{a}=d_{wz}\leq a$ and 
$\overline{b}=d_{wy}\geq b$ (see~\eqref{eq:diyz}) satisfy
\begin{equation}\label{eq:aabb}
	|a-\overline{a}|+|\overline{b}-b|\le 8p^{-6}\delta\,.
\end{equation}
Applying~\eqref{eq:dijk} to the pairs $(w,y)$ and $(w,z)$ in place of $(i,j)$ 
and subtracting the resulting estimates we obtain
\[
	\Big|\sum_{i\in[t]}d_{wi}(\overline{a}d_{zi}-\overline{b}d_{yi})\Big|\le 2\delta t\,.
\]
Thus the triangle inequality and~\eqref{eq:dsums} lead to
\begin{align*}
	|a\overline{a}-b\overline{b}|\cdot p^3t/2 &\le 
	\Big|\sum_{i\in [t]} d_{wi}(a\overline{a}-b\overline{b})\Big| \\
	&\le \Big|\sum_{i\in [t]} d_{wi}\overline{a}(a-d_{zi})\Big|
	+\Big|\sum_{i\in [t]} d_{wi}\overline{b}(d_{yi}-b)\Big|
	+\Big|\sum_{i\in[t]}d_{wi}(\overline{a}d_{zi}-\overline{b}d_{yi})\Big| \\
	&\le 2(1+4p^{-6})\delta t\,,
\end{align*}
i.e., $|a\overline{a}-b\overline{b}|\le 4p^{-3}(1+4p^{-6})\delta$. 
Consequently~\eqref{eq:aabb} yields
\[
	p^3(a-b)\le a^2-b^2\le |a\overline{a}-b\overline{b}|+a|a-\overline{a}|+b|\overline{b}-b|	
	\le 4p^{-3}(1+4p^{-6})\delta+8p^{-6}\delta
\]
and thus $a-b\le 4p^{-6}(1+2p^{-3}+4p^{-6})\delta\le \rho$. 

Now for any four indices $i,j,k,\ell\in[t]$ the extremal choice of $a-b$ gives 
\[
	|d_{ij}-d_{k\ell}|\le |d_{ij}-d_{i\ell}|+|d_{i\ell}-d_{k\ell}|\le 2(a-b)\le 2\rho\,.
\] 
In other words, there is an interval of length $2\rho$ containing all the $d_{ij}$.
In the light of~\eqref{eq:dijk} and the smallness of $\rho$ this interval needs to be 
contained in $(p-\eps, p+\eps)$. Thereby Lemma~\ref{lem:2sets-reduced} is proved.
\end{proof}

As we have already said, our proof of Theorem~\ref{thm:2sets} depends on Szemer\'edi's 
regularity lemma~\cite{Sz78}, a version of which we would like to state next.

\begin{thm}[Regularity lemma]\label{lem:RL}
For every positive real number $\delta$ there is a positive integer~$T$ such that every
graph $G=(V,E)$ admits a partition $V=V_0\dcup V_1\dcup\ldots\dcup V_t$ of its vertex set 
obeying the following conditions:
\begin{enumerate}[label=\alabel]
\item \label{it:RLa} $t\le T$, $|V_0|\le \delta\,|V|$, and $|V_1|=\ldots=|V_t|>0$.
\item \label{it:RLb} For each $i\in[t]$ there are at most $\delta t$ many indices $j\in[t]$ 
						such that the pair $(V_i, V_j)$ is not $\delta$-quasirandom.\qed
\end{enumerate}
\end{thm}

Here a pair $(A, B)$ of nonempty subsets of $V$, say with density 
$d=\frac{e(A, B)}{|A|\,|B|}$, is said to be~{\it $\delta$-quasirandom} 
if we have $e(X, Y)=d\,|X|\,|Y|\pm\delta\,|A|\,|B|$ for all $X\subseteq A$ and $Y\subseteq B$. 

The above statement differs in several aspects from the ``standard''  
regularity lemma and we briefly discuss those differences:
\begin{itemize}
\item  The crucial 
property obtained for most pairs $(V_i, V_j)$ is often taken to be something called 
$\delta$-regularity\footnote[1]{A pair $(A, B)$ of subsets of $G$ is said to be 
$\delta$-regular if $|d(X, Y)-d(A, B)|\le\delta$ holds for all 
$X\subseteq A$ and~$Y\subseteq B$ with $|X|\ge \delta|A|$ and $|Y|\ge\delta |B|$.} 
rather than $\delta$-quasirandomness. These two concepts are known 
to be equivalent up to polynomial losses in the involved constants, and in fact 
$\delta$-regularity implies $\delta$-quasirandomness. Our reason for working with this 
notion here is that allows a slightly cleaner presentation of the proof.
\item Instead of the second condition one usually finds a weaker clause just stating that at 
most $\delta t^2$ pairs $(V_i, V_j)$ fail to be quasirandom in the literature. 
The above version has also been used and
to obtain it, one may apply the standard version of the regularity lemma with some appropriate 
$\delta'\ll\delta$ in place of $\delta$ and then relocate all classes~$V_i$ 
with~$i>0$ that violate~\ref{it:RLb} to~$V_0$.
\item Usually one requires also a lower bound $t_0$ on the number of vertex classes~$t$ in advance 
and then one obtains $T\ge t\ge \min(t_0, |V|)$ rather than just $T\ge t$
in the first part of~\ref{it:RLa}. 
The rationale behind this is that in many applications one has no 
intentions of ``looking inside the individual $V_i$,'' wherefore it brings certain advantages
to have these sets reasonably small. In our current situation, however, even the extreme 
outcome $t=1$ would be useful. In view of~\ref{it:RLb} it would mean that the pair 
$(V_1, V_1)$ is $\delta$-quasirandom, and since, provided that $\delta$ is small, 
$V_1$ would be almost all of $V(G)$, this is essentially all we need to infer for the proof of Theorem~\ref{thm:2sets}.
\end{itemize}

In the course of proving Theorem~\ref{thm:2sets} we will also need to be able to count 
triangles after regularising $G$. This will be rendered by the following strong, but 
well-known, form of the triangle counting lemma. 

\begin{lemma}[Triangle counting lemmma]\label{lem:CL}
Let $A$, $B$, and $C$ denote three nonempty subsets of~$V(G)$ for some graph $G$. 
Suppose that the pairs $(B, C)$, $(C, A)$, and $(A, B)$ have edge densities $\alpha$, $\beta$, 
and $\gamma$ respectively, and that at least two of these three pairs are 
$\delta$-quasirandom. 
Then $\tri(A, B, C)=(\alpha\beta\gamma\pm 2\delta)|A|\,|B|\,|C|$.\qed
\end{lemma}

We apply Lemma~\ref{lem:2sets-reduced} together with the regularity method in 
form of Theorem~\ref{lem:RL} and Lemma~\ref{lem:CL} and
deduce Theorem~\ref{thm:2sets}.

\begin{proof}[Proof of Theorem~\ref{thm:2sets}] 
We begin by choosing certain auxiliary constants obeying the hierarchy
\[
	\eta\ll T^{-1}\ll \delta\ll\eps, p\,,
\]
where $T$ is the integer obtained by applying the regularity Theorem~\ref{lem:RL} with $\delta$. 
Now let any graph $G=(V, E)$ satisfying
\begin{equation}\label{eq:T2}
	\tri(A, B, V)=p^3\,|A|\,|B|\,|V|\pm\eta\,|V|^3
\end{equation}
for all $A, B\subseteq V$ be given. The regularity lemma yields a partition 
\[
	V(G)=V_0\dcup V_1\dcup\ldots\dcup V_t
\]
satisfying the above clauses~\ref{it:RLa} and~\ref{it:RLb}. 
For $i, j\in [t]$ we denote the density of the pair 
$(V_i, V_j)$ by $d_{ij}$. The assumption~\eqref{eq:T2} is only going to be used in the 
special case $A, B\in\{V_1, \ldots, V_t\}$. It then discloses the following useful property 
of the numbers $d_{ij}$:
\begin{equation}\label{eq:dij}
	d_{ij}\sum_{k\in[t]}d_{ik}d_{jk}=(p^3\pm 9\delta)t\,,
\end{equation}
for all $i$, $j\in[t]$.
To see this, we consider any two indices $i, j\in [t]$. 
Let $R_i$ denote the set of all $k\in[t]$ for which the pair $(V_i, V_k)$ is 
not $\delta$-quasirandom, let $R_j$ be defined similarly with respect to $j$, 
and set $R=R_i\cup R_j$. Owing to condition~\ref{it:RLb} from Theorem~\ref{lem:RL} 
we have $|R_i|\le \delta t$ and $|R_j|\le \delta t$, whence $|R|\le 2\delta t$. 
Let us write $M=|V_1|=\ldots=|V_t|$. Then~$Mt=|V|-|V_0|\ge (1-\delta)\,|V|$. 
As we may assume $\delta\le\tfrac12$, is follows that $|V|\le 2Mt$, 
whence $|V_0|\le 2\delta Mt$. Now we have
\[
	\Big|\tri(V_i, V_j, V)-M^3d_{ij}\sum_{k\in[t]}d_{ik}d_{jk}\Big|
	\le
	\tri(V_i, V_j, V_0)+\sum_{k\in[t]}\big|\tri(V_i, V_j, V_k)-M^3d_{ij}d_{ik}d_{jk}\big|\,.
\]
Here the first term may be estimated trivially by
\[
	\tri(V_i, V_j, V_0)\le |V_0|\,|V_i|\,|V_j|\le 2\delta M^3t\,.
\]
Moreover, for $k\in [t]\setminus R$ the triangle counting lemma (Lemma~\ref{lem:CL}) tells us that
\[
	\big|\tri(V_i, V_j, V_k)-M^3d_{ij}d_{ik}d_{jk}\big|\le 2\delta M^3\,,
\]
while for $k\in R$ we still have the obvious bound
\[
	\big|\tri(V_i, V_j, V_k)-M^3d_{ij}d_{ik}d_{jk}\big|\le M^3\,.
\]
Due to $|R|\le 2\delta t$ all this combines to 
\begin{equation}\label{eq:41-2a}
	\Big|\tri(V_i, V_j, V)-M^3d_{ij}\sum_{k\in[t]}d_{ik}d_{jk}\Big|\le 6\delta M^3t\,.
\end{equation}
On the other hand, plugging $A=V_i$ and $B=V_j$ into~\eqref{eq:T2} we learn
\[
	\big|\tri(V_i, V_j, V)-p^3M^2|V|\big|\le \eta\,|V|^3\,, 
\]
which in turn yields 
\[
	\big|\tri(V_i, V_j, V)-p^3M^3t\big|\le p^3M^2\bigl(|V|-Mt\bigr)+\eta\,|V|^3\,.
\]
In view of 
\[
	M^2\bigl(|V|-Mt\bigr)+\eta\,|V|^3= M^2\,|V_0|+\eta\,|V|^3
	\le M^3\bigl(2\delta t +8\eta t^3\bigr)\le M^3t(2\delta+8\eta T^2)
\]
a suitable choice of $\eta$ leads to 
\[
	\big|\tri(V_i, V_j, V)-p^3M^3t\big|\le 3\delta M^3t\,.
\]
Together with~\eqref{eq:41-2a} this concludes the proof of~\eqref{eq:dij}. 

We may assume that depending on $\eps$ and $p$ the constant $\delta$ has been chosen 
so small that Lemma~\ref{lem:2sets-reduced} guarantees $d_{ij}=p\pm \tfrac\eps 2$ 
for all $i, j\in [t]$.
Let us write $S$ for the set of all pairs~$(i, j)\in [t]^2$ such that the pair $(V_i, V_j)$ 
is not $\delta$-quasirandom. Notice that condition~\ref{it:RLb} from Theorem~\ref{lem:RL} 
implies $|S|\le \delta t^2$. 

Now for any $A, B\subseteq V$ we have
\begin{equation}\label{41-eAB}
	\Big|e(A, B)-p\,|A|\,|B|\Big|
	\le \sum_{i=0}^t\sum_{j=0}^t
		\Big|e(A\cap V_i, B\cap V_j)|-p\,|A	\cap V_i|\,|B\cap V_j|\Big|\,.
\end{equation}
Each term on the left-hand side having $i=0$, $j=0$, or $(i, j)\in S$ may be bounded 
from above by $|V_i|\,|V_j|$, so altogether these terms contribute at most
\[
	|V|^2-\bigl(|V|-|V_0|\bigr)^2+|S|\,M^2\le 2\delta\,|V|^2+\delta M^2t^2\le 3\delta\,|V|^2\,.
\]
Owing to the quasirandomness, each of the remaining terms on the right hand side 
of~\eqref{41-eAB} may be estimated as follows:
\begin{align*}
	\Big|e(A\cap V_i, B\cap V_j) &-p\,|A\cap V_i|\,|B\cap V_j|\Big| \\
	&\le \Big|e(A\cap V_i, B\cap V_j)-d_{ij}\,|A\cap V_i|\,|B\cap V_j|\Big|
			+|d_{ij}-p|\,|V_i|\,|V_j| \\ 
	&\le \bigl(\delta+\tfrac\eps 2\bigr)\,|V_i|\,|V_j|
\end{align*}
So taken together these terms amount to at most $(\delta+\tfrac\eps 2\bigr)\,|V|^2$, 
and in view of $\delta\ll\eps$ we finally we arrive at
\[
	\big|e(A, B)-p\,|A|\,|B|\big|\le \bigl(4\delta+\tfrac\eps 2\bigr)\,|V|^2\le \eps\,|V^2|
\]
for arbitrary $A, B\subseteq V$. This proves that $G$ is indeed $(\eps,p)$-quasirandom.
\end{proof}

\section{Proof of Corollary~\ref{thm:K3Nforcing}}\label{sec:forcing}
In this section we deduce Corollary~\ref{thm:K3Nforcing}. The only thing we need to check 
is the following proposition, which combined with Theorem~\ref{thm:2sets} yields the corollary.

\begin{prop}\label{prop:get2sets}
Suppose that a graph $G=(V, E)$ satisfies 
\[
	t(K_3, G)\ge (1-\delta)p^3 \qand \t(C_4^\triag, G)\le (1+\delta)p^{12}
\]
for some $p, \delta\in (0,1]$. Then
\[
	\tri(A, B, V)=p^3\,|A|\,|B|\,|V|\pm 8\delta^{1/4}p^3\,|V|^3
\]
holds for all $A, B\subseteq V$.
\end{prop}

Besides the Cauchy--Schwarz inequality itself the proof will also use the following known 
and easy to confirm result on situations where equality almost holds.

\begin{fact}\label{CSI=}
Let $x_1, \ldots, x_n$, $\alpha$, and $\nu$ denote any real numbers satisfying
\[
	\sum_{i=1}^n x_i=\alpha n 
	\qand
	\sum_{i=1}^n x^2_i=(\alpha^2+\nu^2)n\,.
\]
Then we have $\sum_{i=1}^m x_i=\alpha m\pm \nu n$ for any $m=0,1,\ldots, n$.\qed
\end{fact}

\begin{proof}[Proof of Proposition~\ref{prop:get2sets}]
Let $G=(V,E)$, $p$, and $\delta$ be as in the hypothesis and let $A, B\subseteq V$ be arbitrary. 

\parpic[l]{\includegraphics[width=2.2cm]{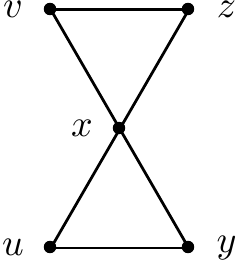}} 
We begin the proof by setting up some notation: Let $S$ denote the graph obtained by gluing two 
triangles together at a vertex (see the picture on the left). We consider the homomorphism densities of 
$K_3$, $S$, and $C_4^\triag$ and 
define the the real numbers $a$, $b$, and $c$ so as to obey
\begin{align}
	\hom(K_3, G)&=ap^3n^3\,,  \label{eq:countK3} \\ 
	\hom(S , G)&=bp^6n^5\,,  \label{eq:countS} \\
	\intertext{and} 
	\quad \hom(C_4^\triag, G)&=cp^{12}n^8\,, \label{eq:countN}
\end{align}
where $n=|V|$. Notice that the assumption translates to $a\ge 1-\delta$ and $c\le 1+\delta$. 
Given a vertex $x\in V$ we write $T_x$ for the number of pairs $(y,z)\in V^2$ 
such that $xyz$ is a triangle in~$G$. Moreover, for any two 
vertices $u, v\in V$ we denote the number of triples $(x,y,z)\in V^3$ with
$ux, uy, vx, vz, xy, xz\in E$ by $S_{u,v}$ (see figure). In terms of these numbers the 
equations~\eqref{eq:countK3}, \eqref{eq:countS}, and \eqref{eq:countN} rewrite as
\begin{align}     
	\sum_{x\in V}T_x&=ap^3n^3\,, \label{eq:Tsum} \\
	\sum_{x\in V}T^2_x=\sum_{(u,v)\in V^2}S_{u,v}&=bp^6n^5\,, \label{eq:TSos} \\
	\intertext{and} 
	\quad \sum_{(u,v)\in V^2}S^2_{u,v}&= cp^{12}n^8\,. \label{eq:SSos}
\end{align}
Thus the Cauchy--Schwarz inequality implies
\begin{equation}\label{eq:abc}
	(1-\delta)^4\le a^4\le b^2\le c\le 1+\delta\,.
\end{equation}
Due to Fact~\ref{CSI=} (applied with $\alpha=ap^3n^2$ and $\nu=p^3n^2\sqrt{b-a^2}$), \eqref{eq:Tsum}, and \eqref{eq:TSos} we have
\[
	\tri(A, V, V)=\sum_{x\in A} T_x=ap^3n^2\,|A|\pm \sqrt{b-a^2}\cdot p^3n^3\,.
\]
Owing to  $1-\delta\leq a\le 1+\delta$ and $\sqrt{b-a^2}\le 3\sqrt{\delta}$ (see~\eqref{eq:abc}),
this leads to
\begin{equation}\label{eq:TAVV}
	\tri(A, V, V)=p^3n^2\,|A|\pm 4\delta^{1/2}p^3n^3\,.
\end{equation}
Similarly Fact~\ref{CSI=} (applied with $\alpha=bp^6n^3$ and $\nu=p^6n^3\sqrt{c-b^2}$), \eqref{eq:TSos}, and \eqref{eq:SSos} give
\[
	\sum_{(u,v)\in A^2}S_{u,v}\le bp^6n^3\,|A|^2+\sqrt{c-b^2}\cdot p^6n^5
	\le 
	\bigl((1+\delta)\,|A|^2+3\delta^{1/2}n^2\bigr) p^6n^3\,,
\]
whence
\begin{equation}\label{eq:TAVV2}
	\sum_{y\in V}\tri(A, y, V)^2
	=
	\sum_{(u,v)\in A^2}S_{u,v}
	\le 
	(|A|^2+4\delta^{1/2}n^2)p^6n^3\,.
\end{equation}
Now if $Q$ and $R$ denote the real numbers satisfying
\[
	\sum_{y\in V}\tri(A, y, V)=p^3n^2Q 
	\quad \text{ and } \quad
	\sum_{y\in V}\tri(A, y, V)^2=p^6n^3R\,,
\]
then \eqref{eq:TAVV} and \eqref{eq:TAVV2} entail
\begin{equation}\label{eq:QR}
	Q=|A|\pm 4\delta^{1/2}n
	\quad \text{ and } \quad
	R\le |A|^2+4\delta^{1/2}n^2\,,
\end{equation}
whilst a final application of Fact~\ref{CSI=} (applied with $\alpha=p^3nQ$ and $\nu=p^3n\sqrt{R-Q^2}$) reveals
\begin{equation}\label{eq:TABV}
	\tri(A, B, V)=p^3nQ\,|B|\pm p^3n^2\sqrt{R-Q^2}\,.
\end{equation}
It follows from~\eqref{eq:QR} that
$R-Q^2\le 12\delta^{1/2}n^2$ and
owing to \eqref{eq:TABV} 
we obtain
\begin{multline*}
	\Big|\tri(A, B, V)-p^3\,|A|\,|B|\,n\Big|\le p^3n\,|B|\,\big|Q-|A|\big|+4\delta^{1/4}p^3n^3 \\
	\le p^3n^3(4\delta^{1/2}+4\delta^{1/4})\le 8\delta^{1/4}p^3n^3\,,
\end{multline*}
as desired.
\end{proof}

\section{Proof of Corollary~\ref{thm:Hauptsatz}}\label{sec:ana}

\subsection{Notation} We mostly follow the notation from Lov\'asz's research 
monograph~\cite{Lov12} and in this subsection we remind the reader of what we 
actually need. 
By $\cW$ we mean the space of all bounded symmetric measurable functions from 
the unit square $[0, 1]^2$ to the set of reals.
So $\cW$ is a linear space whose members are sometimes referred to as {\it kernels}. 
It is known that for each kernel $W$ the maximum
\[
	\| W\|_\Box=
	\max\biggl\{ 
		\Big|\int_{A\times B}W(x, y)\;\mathrm{d}x\,\mathrm{d}y\,\Big| \colon
		A, B\subseteq [0, 1] \text{ measurable}
	\biggr\} 
\]
exists and that $W\longmapsto \|W\|_\Box$ is a norm on $\cW$, the so-called
{\it cutnorm}. If~$\|W\|_\Box=0$ holds for $W\in\cW$, then
this kernel vanishes almost everywhere, i.e., the set $\{(x, y)\in [0, 1]^2\colon W(x, y)\ne 0\}$
has measure zero (see~\cite{Lov12}*{Section~8.2.3}). 

The group of measure preserving bijections form the unit interval onto itself is denoted 
by~$\cS_{[0, 1]}$. This group acts in an obvious way on the space of kernels by
\[
	W^\phi(x, y)=W(\phi(x), \phi(y)) 
\]
for all $W\in\cW$, $\phi\in \cS_{[0, 1]}$,  and $x, y\in [0, 1]$.
The {\it cut distance} $\delta_\Box(W_1, W_2)$ between two kernels $W_1$, $W_2\in\cW$ is defined by
\[
	  \delta_\Box(W_1, W_2)
	  =\inf\bigl\{
	  			\|W_1-W_2^\phi\|_\Box \colon \phi\in\cS_{[0, 1]}
			\bigr\}\,.
\]
Actually, this infimum is known to be a minimum, but this fact is rarely needed 
and we shall make no use of it.
 
Those $W\in\cW$ that satisfy $W(x, y)\in [0, 1]$ for all $x, y\in [0, 1]$ are 
called {\it graphons} and the set of all graphons is denoted by $\cW_0$. 
With each graph $G=(V,E)$ we can associate a graphon~$W_G$ by taking an arbitrary partition 
$[0, 1]=\bigdcup_{v\in V}P_v$ of the unit interval into measurable pieces of 
measure $|V|^{-1}$ and defining for all $x, y\in [0, 1]$
\[
	W_G(x, y)=
		\begin{cases}
		1, & \text{ if $uv\in E$, where $x\in P_u$, $y\in P_v$} \cr
		0, & \text{ else.}
		\end{cases}
\]
This graphon depends, of course, not only on $G$ but also on the 
underlying partition, but modulo the action of $\cS_{[0, 1]}$ mentioned above it is uniquely
determined by $G$. 

An important insight due to Lov\'asz and Szegedy~\cite{LoSz} is the compactness of the pseudometric 
space~$(\cW_0, \delta_\Box)$. In fact, the compactness easily implies the regularity lemma for graphs 
(see also~\cite{Lov12}*{Theorem~9.23}). 
This result does actually occupy a central place in the limit theory of dense graphs.
Besides, it is beautifully complemented by the fact that the set 
$\{W_G\colon G \text{ is a graph}\}$ is dense in~$(\cW_0, \delta_\Box)$.

Given a graph $F$ and a kernel $W$ the {\it homomorphism density} $\t(F, W)$ is defined
to be the multidimensional integral
\[
	\t(F, W)=\int_{[0, 1]^{V(F)}}\prod_{uv\in E(F)}W(x_u, x_v)\prod_{u\in V(F)}\mathrm{d}x_u\,.
\]
This stipulation extends the usual definition of homomorphism densities for graphs 
in the sense that $t(F, W_G)=t(F, G)$ holds for all graphs $F$ and $G$.

Analytically speaking, the {\it global counting lemma} asserts that for every graph $F$ the map~$W\longmapsto\t(F, W)$ from~$(\cW_0, \delta_\Box)$ to $[0, 1]$ is Lipschitz continuous 
with Lipschitz constant~$e(F)$ (see~\cite{Lov12}*{Theorem~10.23}).

\subsection{Forcing families}

Let us write $W\equiv p$ for a kernel $W$ and a real number $p$ if $W$ agrees almost 
everywhere with the constant function whose value is always $p$. 

\begin{lemma}\label{lem:folklore}
A pair of graphs $(F_1,F_2)$  is forcing if and only if we have $W\equiv p$ for every real $p\in(0, 1]$ 
and every graphon $W$ with $\t(F_i, W)=p^{e(F_i)}$ for $i=1,2$.
\end{lemma}

\begin{proof}
This is implicit in the discussion from~\cite{Lov12}*{Section~16.7.1}. 
\end{proof}

Theorem~\ref{thm:2sets} is the discrete analogue of the following statement. 

\begin{thm}\label{thm:2sets-ana}
If a graphon $W$ and a real $p\in(0, 1]$ are such that 
\[
	W(x, y)\int_0^1 W(x, z)W(y, z)\,\mathrm{d}z=p^3
\]
holds for almost all $(x, y)\in [0, 1]^2$, then $W\equiv p$.
\end{thm}

One way to show this proceeds by carefully repeating the proof of 
Lemma~\ref{lem:2sets-reduced} in this analytical setting. This is 
not hard, but somewhat technical, and hence we would like to present 
an alternative argument here. 

\begin{proof}[Proof of Theorem~\ref{thm:2sets-ana}]
Define a graphon $U$ by
\begin{equation}\label{eq:dfn-U}
	U(x, y)=W(x, y)\int_0^1 W(x, z)W(y, z)\,\mathrm{d}z
\end{equation}
for all $x, y\in [0, 1]$. Now the assumption $U\equiv p^3$ 
leads to $t(K_3, W)=t(K_2, U)=p^3$ and $t(C_4^\triag, W)=t(C_4, U)=p^{12}$.
Since the pair $(K_3, C_4^\triag)$ is forcing, it follows by Lemma~\ref{lem:folklore}
that we have indeed $W\equiv p$.
\end{proof}

We conclude this section with the proof of our main result on forcing pairs involving triangles.

\begin{proof}[Proof of Corollary~\ref{thm:Hauptsatz}]
Suppose that $W$ is a graphon and $p\in (0, 1]$ is a real number such that 
\begin{equation}\label{eq:hypothesis}
	\t(K_3, W)=p^3 \qand \t(F^\triag, W)=p^{3e(F)}\,.
\end{equation}
In view of Lemma~\ref{lem:folklore} we have to prove that $W\equiv p$. To this end, 
we look again at the graphon $U$ defined by~\eqref{eq:dfn-U}.
The hypothesis~\eqref{eq:hypothesis} rewrites in terms of $U$ as
\[
	\t(K_2, U)=p^3 \qand \t(F, U)=p^{3e(F)}\,.
\]
Since the pair $(K_2, F)$ is forcing, it follows that $U\equiv p^3$,
and in the light of Theorem~\ref{thm:2sets-ana} we get indeed $W\equiv p$.
\end{proof}

\section{Concluding Remarks}\label{sec:conc}
We close with a few remarks and open problems for future research.
\begin{itemize}
\item Corollary~\ref{thm:Hauptsatz} raises the general problem to characterise 
all graphs~$F$ with the property that the pair $(K_3, F)$ is forcing. However, given our current state of
knowledge and the fact that this is still open for $(K_2,F)$ 
it appears unclear how to even formulate a plausible 
conjecture in this regard.

\item The proof of Theorem~\ref{thm:2sets} presented in Section~\ref{sec:2sets} is based on 
Szemer\'edi's regularity lemma and as a consequence this proof requires 
that~$\eta^{-1}$ behaves like an exponential tower of height $\textrm{poly}(\eps^{-1},p^{-1})$. 
We would like to
thank L.~M.~Lov\'asz for pointing out to us that a different argument utilising the 
Frieze--Kannan regularity lemma~\cite{FK99} shows that $\eta^{-1}=2^{(\eps p)^{-\Theta(1)}}$
would suffice as well. To see this one exploits that the assumption of Theorem~\ref{thm:2sets}
implies $t(K_3, G)\approx p^3$ and $t(C_4^\triag, G)\approx p^{12}$. Due to the global counting lemma
this gives us two approximate equalities for the densities~$d_{ij}$ arising in a Frieze--Kannan 
regular partition of $G$. As in the proof of Proposition~\ref{prop:get2sets}
two reverse applications of the Cauchy--Schwarz inequality then entail that the 
assumption~\eqref{eq:dijk} of Lemma~\ref{lem:2sets-reduced} holds with at most $o(t^2)$ 
exceptions. Working a little bit harder in the proof of this lemma it can then be shown that 
this is enough to imply that there are at most $o(t^2)$ pairs $(i, j)\in [t]^2$ for which 
$d_{ij}\approx p$ fails. This, however, is in turn equivalent to $G$ being $p$-quasirandom.

Recently, it was shown by Conlon, Fox, and Sudakov~\cite{CFS}
that the corresponding dependency of the parameters in the Simonovits--S\'os theorem for the triangle is in fact linear 
(see also~\cite{He,RSch} for further results). In view of these results, it seems an interesting open question 
whether Theorem~\ref{thm:2sets} holds also for~$\eta=\textrm{poly}(\eps,p)$.

\item It appears to be an intriguing open problem to find the appropriate generalisation of 
Theorem~\ref{thm:2sets} (and Corollary~\ref{thm:K3Nforcing})
for graphs other than the triangle. At this point even for cliques $K_k$ with $k\geq 4$ 
this is an open problem. For integers $1\leq \l\leq k$ we say~\emph{$K_k$ is $\l$-forcing}, if every graph 
$G=(V,E)$ satisfying for all subsets~$X_1,\dots,X_\l\subseteq V$
\[
	\cK_k(X_1,\dots,X_\l)=p^{\binom{k}{2}}|V|^{k-\l}\prod_{i=1}^{\l}|X_i|+o(|V|^k)
\]
for some $p\in(0,1]$ is $p$-quasirandom, where $\cK_k(X_1,\dots,X_\l)$ denotes the number of $k$-tuples 
\[
	(v_1,\dots,v_k)\in X_1\times\dots\times X_\l\times V^{k-\l}
\] 
that span a $K_k$ in $G$. The Simonovits--S\'os theorem implies for every $k\geq 2$ that~$K_k$ is $k$-forcing and it is not hard to show that no clique is $1$-forcing.
Theorem~\ref{thm:2sets} tells us that $K_3$ is $2$-forcing and it would be interesting to 
determine for every $k\geq 4$ the smallest $\l$ such that $K_k$ is $\l$-forcing. 
The proof of Theorem~\ref{thm:2sets} can be adjusted to show that $\l=\lceil\frac{k+1}{2}\rceil$
suffices and this was also noted independently by Hubai et al.~\cite{HKPP}.
Currently, we are not aware of any reason that rules out the possibility that every clique $K_k$
is $2$-forcing or that there is a universal bound independent of~$k$.

\item One may also consider hypergraphs extensions of those results. For example, one may investigate, whether
the tetrahedron $K_4^{(3)}$ is $3$-forcing for the notion of quasirandomness investigated in~\cite{CG90,KRS02}.

\end{itemize}

\end{document}